\documentclass{article}
\usepackage{amsfonts}
\usepackage{amsthm}
\usepackage[pdftex]{graphicx}
\pdfoutput=1
\usepackage{latexsym}
\usepackage{amssymb}
\usepackage{amstext}
\usepackage{amsmath}
\usepackage{amssymb}
\usepackage{graphicx}
\usepackage[colorlinks=true,citecolor=black,linkcolor=black,urlcolor=blue]{hyperref}
\theoremstyle{plain}
\newtheorem{example}{Example}
\newtheorem{lemma}{Lemma}
\newtheorem{theorem}{Theorem}
\newtheorem{corollary}{Corollary}
\newtheorem{definition}{Definition}
\newtheorem{question}{Question}
\newtheorem{proposition}{Proposition}

\newcommand{\mbf}{\mathbf} 
\newcommand{\supp}{\text{supp}} 
\newcommand{\wt}{\text{wt}} 
\newcommand{\Aut}{\text{Aut}}
\newcommand{\wtO}{\widetilde\Omega_{2k}}
\newcommand\1{\bigskip\noindent}
\newcommand\Dist{{\rm Dist}}
\newcommand\Det{{\rm Det}}
\newcommand\bv{{\mbf v}}
\newcommand\bu{{\mbf u}}

\newcommand\one{{\mbf 1}}

\author{Debra Boutin\\
\small Department Mathematics and Statistics\\[-0.8ex]
\small Hamilton College\\[-0.8ex] 
\small Clinton, NY, U.S.A.\\
\small\tt dboutin@hamilton.edu\\
\and
Sally Cockburn \\
\small Department Mathematics and Statistics\\[-0.8ex]
\small Hamilton College\\[-0.8ex] 
\small Clinton, NY, U.S.A.\\
\small\tt scockbur.hamilton.edu}

\begin{document}

\title{Distinguishing Orthogonality Graphs}
\author{Debra Boutin and Sally Cockburn}
\maketitle    

\begin{abstract}    A graph $G$ is said to be {\it $d$-distinguishable} if there is a labeling of the vertices with $d$ labels so that only the trivial automorphism preserves the labels.  The smallest such $d$ is the {\it distinguishing number}, $\Dist(G)$. A set of vertices $S \subseteq V(G)$ is a {\it determining set} for $G$ if every automorphism of $G$ is uniquely determined by its action on $S$.  The size of a smallest determining set for $G$ is called the {\it determining number}, $\Det(G)$. The orthogonality graph $\Omega_{2k}$ has vertices which are bitstrings of length $2k$ with an edge between two vertices if they differ in precisely $k$ bits. This paper shows that $\Det(\Omega_{2k}) = 2^{2k-1}$ and that, if $\binom{m}{2} \geq 2k$ when $k$ is odd or $\binom{m}{2}\geq 2k+1$ when $k$ is even, then $2< \Dist(\Omega_{2k}) \leq m$.
\end{abstract}


\section{Introduction}

A labeling of the vertices of a graph $G$ with the integers $1,\ldots, d$ is called a {\it $d$-distinguishing labeling} if no nontrivial automorphism of $G$ preserves the labels. A graph is called {\it $d$-distinguishable} if it has a $d$-distinguishing labeling.  The {\it distinguishing number} of $G$, ${\rm Dist}(G)$, is the fewest number of labels necessary for a distinguishing labeling.  Albertson and Collins  introduced graph distinguishing in \cite{AC}.  Over the last few decades, this topic has generated significant interest and abundant results.  See, for instance, \cite{AlSo2019a, JaMuBe2019, Ji2018, MeSo2019b} for recent works.

Most of the work in the last few decades has been in studying large families of  graphs and showing that all but a finite number in each family have distinguishing number 2.   Examples of this for finite graphs include: hypercubes $Q_n$ with $n\geq 4$ \cite{BC}, Cartesian powers  $G^n$ for a connected graph $G\ne K_2,K_3$ and $n\geq 2$ \cite{A, IK2,KZ}, Kneser graphs $K_{n:k}$ with  $n\geq 6, k\geq 2$ \cite{AB2}, and (with seven small exceptions) $3$-connected planar graphs \cite{FNT}.  Examples for  infinite graphs include:  the  denumerable random graph \cite{IKT}, the infinite hypercube \cite{IKT}, locally finite trees with no vertex of degree 1 \cite{WZ}, and denumerable vertex-transitive graphs of connectivity 1 \cite{STW}. 

Exhaustion shows that the cycles $C_3, C_4, C_5$ and the hypercubes $Q_2, Q_3$ each have distinguishing number 3.  Some  infinite graph families that are not 2-distinguishable are $K_n$ (${\rm Dist}(K_n)=n$) and the complete bipartite graph $K_{m,n}$ (${\rm Dist}(K_{m,n}) = \max\{m,n\}$ for $m\ne n$, and ${\rm Dist}(K_{n,n}) = n+1$). We will see in Section \ref{sec:lowerbd} that orthogonality graphs $\Omega_{2k}$ are also not 2-distinguishable.

A useful tool used in finding the distinguishing number is the {\it determining set} \cite{B1}, a set of vertices whose pointwise stabilizer is trivial. The {\it determining number} of a graph $G$, $\Det(G)$, is the size of a smallest determining set. For some families we only have bounds on the determining number.  For instance, for the Kneser graph, $\log_2 (n+1)\leq \Det(K_{n:k}) \leq n-k$ with both upper and lower bounds sharp \cite{B1}. However, there are families for which we know the determining number exactly.  For instance in Cartesian powers, $\Det(Q_n) = \lceil \log_2 n \rceil +1$, and $\Det(K_3^n) = \lceil \log_3(2n+1) \rceil +1$ \cite{B4}. 

The determining set and the distinguishing number were introduced at different times, by different authors, and for distinct purposes. However, Albertson and Boutin connected them in \cite{AB2}, by noting that if $G$ has a determining set of size $d$, then there is a $(d+1)$-distinguishing labeling for $G$. Thus $\Dist(G) \leq \Det(G) +1$.  We will find this relationship useful in pursuing the distinguishing number of orthogonality graphs.

For a positive integer $k$, the orthogonality graph $\Omega_{2k}$ has as vertices all bitstrings of length $2k$, with two vertices adjacent if their Hamming distance is $k$. So $\Omega_{2k}$ has the same vertices as the $2k$-dimensional hypercube, but with two vertices adjacent if their bitstrings are orthogonal when considered as vectors of ${\mathbb Z}_2^{2k}$. The graph $\Omega_{2^r}$ is used in quantum information theory to study the cost of simulating a specific quantum entanglement on $r$ qubits. See \cite{BrClTa1999,BuClWi1999,NiCh2000} for an introduction to quantum information theory and for the specific situation that properly coloring $\Omega_{2^r}$ addresses. Inspired by this quantum situation, Godsil and Newman studied the independence and chromatic numbers of $\Omega_{2^r}$ in \cite{GoNe2008}.  Earlier, Ito \cite{Ito1985} studied the even component of $\Omega_{4k}$, calling it the Hadamard graph of size $4k$,
and investigated its maximal complete subgraphs, its spectrum, and bounds on its chromatic number. Frankl \cite{Fr1986} found the independence number of such graphs and used it to improve bounds on the chromatic number in the case when $k$ is an odd prime power.  In this paper we study the determining and distinguishing numbers of the more general $\Omega_{2k}$.

The paper is organized as follows.  Definitions and facts about determining sets, distinguishing labelings, and orthogonality graphs are given in Section \ref{sect:basics}.  Section \ref{sec:lowerbd}  examines pairs of twin vertices in $\Omega_{2k}$, proves $\Det(\Omega_{2k})=2^{2k-1}$, and shows that $\Omega_{2k}$ is not 2-distinguishable. Section \ref{sec:structure} discusses odd and even vertices in $\Omega_n$, and introduces a quotient graph $\wtO$.  Section \ref{sec:DetQuotient} shows that $\Det(\wtO) \leq 2k-1$.  Finally, Section \ref{sec:upperbd}  provides the upper bound for  $\Dist(\Omega_{2k})$, and Section \ref{sect:questions}  provides some open problems for future work.

\section{Background}\label{sect:basics}

\subsection{Determining Sets and Distinguishing Labelings}

Let $G$ be a graph. A subset $S\subseteq V(G)$ is said to be a {\it determining set} for $G$ if whenever $\varphi, \psi \in {\rm Aut}(G)$ so that $\varphi(x)=\psi(x)$ for all $x\in S$, then $\varphi=\psi$.   Thus every automorphism of $G$ is uniquely determined by its action on the vertices of a determining set.  A determining set is an example of a {\it base of a permutation group action}.  Every graph has a determining set since a set containing all but one vertex of the graph is determining.  The {\it determining number} of  $G$, $\Det(G)$, is the minimum size of a determining set for  $G$.

Recall that the {\it set stabilizer} of $S\subseteq V(G)$ is the set of all  $\varphi\in{\rm Aut}(G)$ for which $\varphi(x) \in S$ for all $x\in S$.  In this case we say that $S$ is invariant under $\varphi$ and we write $\varphi(S)=S$.  The {\it pointwise stabilizer} of $S$ is the set of all $\varphi\in {\rm Aut}(G)$ for which $\varphi(x)=x$ for all $x\in S$. It is easy to see that $S\subseteq V(G)$ is a determining set for $G$ if and only if the pointwise stabilizer of $S$ is trivial.

A labeling $f:V(G) \to \{1,\ldots, d\}$ is said to be {\it $d$-distinguishing} if only the trivial automorphism preserves the label classes.  Every graph has a distinguishing labeling since each vertex can be assigned a distinct label.   A graph is called {\it $d$-distinguishable} if it has a $d$-distinguishing labeling.  The distinguishing number of $G$, ${\rm Dist}(G)$, is the fewest number of labels necessary for a distinguishing labeling. 

\begin{lemma}\label{lem:EquivDist} \rm 
Let $G$ be a graph, $\alpha \in {\rm Aut}(G)$, and $f:V(G) \to \{1, \dots, d\}$ a vertex labeling.  Then $f$ is $d$-distinguishing  if and only if $f \circ \alpha$ is  $d$-distinguishing.
\end{lemma}

\begin{proof} It is straightforward to verify that  $\varphi \in \Aut(G)$ preserves the label classes of  $f \circ \alpha$ if and only if  $\alpha \circ \varphi \circ \alpha^{-1}$ preserves the label classes of $f$.\end{proof}

The following ties together determining sets and distinguishing labelings and facilitates the work in this paper.

\begin{theorem} \label{thm:distdet} \rm \cite{AB2} $G$ is $d$-distinguishable if and only if it has a determining set $S$ of size $d{-}1$ that can be labeled in such a way that any automorphism of ${\rm Aut}(G)$ that preserves the labeling classes of $S$ fixes $S$ pointwise. \end{theorem}

\begin{corollary} \label{cor:distdet} \rm $\Dist(G) \leq \Det(G)+1$. \end{corollary}

\begin{proof} Suppose $S$ is a smallest determining set for $G$ of size $d$.  Label each of the vertices of  $S$ with a different label. Label each of the remaining vertices of $G$ with the label $d+1$. If $\varphi\in {\rm Aut}(G)$ preserves the label classes, then $\varphi$ fixes each of the $\Det(G)$ differently labeled vertices in $S$.  Since $S$ is a determining set, this means $\varphi$ is the identity.  Thus our labeling is a $(d+1)$-distinguishing labeling for $G$. \end{proof}

As we'll see in Lemma \ref{lem:DisDetTwins} below, twin vertices play a significant role in the study of graph symmetry.  

\begin{definition} \rm Two vertices  $u, v \in V(G)$ are called {\it twins} if they have identical sets of neighbors.  That is, $u$ and $v$ are twins if $N(u)=N(v)$.\end{definition}

\begin{lemma}\label{lem:DisDetTwins}\rm Let $u,v\in V(G)$ be twins.   Then the function on $V(G)$ that interchanges $u$ and $v$ and acts as the identity on all other vertices is a graph automorphism.  Thus in any distinguishing labeling members of a twin pair  must have different labels, and further, any determining set must contain at least one member of each twin pair. \end{lemma}

The proof is elementary.

\subsection{Orthogonality Graphs}\label{sec:definition}

\begin{definition} \rm The {\it orthogonality graph} $\Omega_{2k}$ has as its vertex set all bitstrings of length $2k$, 
\[
V(\Omega_{2k}) =\big  \{ \mbf u = u_1 u_2 \dots u_{2k} \mid u_i \in \{0, 1\} \big\} = \mathbb Z_2^{2k},
\]
with two vertices adjacent if the corresponding bitstrings differ in exactly $k$ bits. Let $\mbf 0, \mbf 1$ denote the bitstring of all 0s and all 1s respectively.  Define $\mbf u + \mbf w = (u_1+w_1)(u_2+w_2)  \dots (u_{2k}+w_{2k})$
where all bit-sums are taken modulo 2.  Note that $\Omega_{2k}$ has order $2^{2k}$ and is $\binom{2k}{k}$-regular.
\end{definition}

\begin{example} \rm The smallest orthogonality graph $\Omega_2$ occurs when $k=1$ and is isomorphic to $C_4$. 
\end{example}

\begin{example}\label{ex:O4} \rm  The orthogonality graph $\Omega_4$ is a  6-regular graph of order 16 and  consists of two isomorphic components, each of which is a copy of the circulant graph $C_8[1,2,3]$. \end{example}

\begin{definition}\rm The (Hamming) weight of $\mbf u \in V (\Omega_{2k})$, denoted $\wt(\mbf u)$,  is the number of 1s in its bitstring. Let $\mbf 0$ and $\mbf 1$ be the bitstrings of length $2k$ of weight $0$ and $2k$ respectively. The {\it support} of $\mbf u$ is the set of indices of the bits where its 1's occur.  That is, $\supp(\mbf u) = \{ i \mid u_i = 1\} \subseteq \{1, 2, \dots , 2k\}.$\end{definition}

Note that $\wt(\mbf u) =  |\supp(\mbf u)|$. Also, note that for any $\mbf u, \mbf v \in V(\Omega_{2k})$,
$\supp(\mbf u + \mbf v) = \supp(\mbf u)\,  \triangle \, \supp(\mbf v),$ where $\triangle$ denotes the symmetric difference.  In particular,  $\supp(\mbf u + \mbf 1)$ is the complement of $\supp(\mbf u)$.  Further, $\mbf u, \mbf v \in V(\Omega_{2k})$ are adjacent if and only if $\wt(\mbf u + \mbf v) =|  \supp(\mbf u)\,  \triangle \, \supp(\mbf v)| = k.$

It is easy to verify that the vertex maps described below are automorphisms of $\Omega_{2k}$. 

\1$\bullet$ {\bf Permutation automorphisms.} For any permutation  $\sigma \in S_{2k}$,  let $\sigma$ act on vertices of $\Omega_{2k}$ by permuting the order of the bits; that is,
\[
\sigma(\mbf u) = \sigma(  u_1 u_2 \dots u_{2k} ) = u_{\sigma(1)} u_{\sigma(2)} \dots u_{\sigma(2k)}.
\] 
\1$\bullet$ {\bf Translation automorphisms.} For $\mbf u \in V(\Omega_{2k})$, define $\tau_{\mbf u}:V(\Omega_{2k}) \to V(\Omega_{2k})$  by 
\[
\tau_{\mbf u} (\mbf w) = \mbf u + \mbf w.\]

For $k\geq 2$, these  two families of automorphisms do not exhaust $\Aut(\Omega_{2k})$.  For example, we will see in Section \ref{sec:lowerbd} that  there is an automorphism $\pi_{\mbf 0}$ that transposes $\mbf 0$ 
and $\one$ and leaves all other vertices fixed. The following argument shows that $\pi_{\mbf 0}$ is not in the subgroup generated by permutation automorphisms and translation automorphisms.  

Any composition of translation automorphisms is itself a translation automorphism; the same is true for permutation automorphisms. Note also that for all $\mbf u, \mbf w\in V(\Omega_{2k})$ and $\sigma\in S_{2k}$,
\[
(\sigma \circ \tau_{\mbf u})(\mbf w) = \sigma(\mbf u + \mbf w) = \sigma(\mbf u) + \sigma(\mbf w) =  (\tau_{\sigma(\mbf u)}\circ \sigma) (\mbf w).
\]
  
Thus any automorphism in the subgroup generated by permutations and translations can be written in the form $\tau_{\mbf u} \circ \sigma$.  If $\tau_{\mbf 0}$ is in this subgroup, then there exists  $\mbf u \in V(\Omega_{2k})$ and $\sigma \in S_{2k}$ so that $\pi_{\mbf 0} = \tau_{\mbf u} \circ \sigma$. Note that $\pi_{\mbf 0} (\mbf 0) =\mbf 1$ while $\tau_{\mbf u} \circ \sigma (\mbf 0) = \tau_{\mbf u}(\mbf 0) = \mbf u$. Thus $\mbf u = \mbf 1$. 

However, $\pi_{\mbf 0}$ fixes all vertices other than $\mbf 0$ and $\mbf 1$ while $\tau_{\mbf 1} \circ \sigma$ takes vertices of weight 1 to vertices of weight $2k-1$.  Since $k > 1$, this shows $\pi_{\mbf 0}$ is not in this subgroup. 

Orthogonality graphs are highly symmetric in the sense that they are arc-, edge-, and vertex-transitive.  To show this, suppose $(\mbf x, \mbf y)$ and $(\mbf u, \mbf w)$ are arcs (directed edges) of $\Omega_{2k}$.   Since $(\mbf x, \mbf y)$ and $(\mbf u, \mbf w)$ are edges, each of  $\mbf x + \mbf y$ and $\mbf u + \mbf w$ has weight $k$.  Since their supports have the same size, there is a permutation $\sigma \in S_{2k}$ taking the support of $\mbf x + \mbf y$ to the support of $\mbf u + \mbf w$.  Denote by $\sigma$ the corresponding permutation automorphism of $\Omega_{2k}$.  Then $\sigma(\mbf x+\mbf y) = \mbf u + \mbf w$.  Now consider the automorphism $\tau_\mbf{u} \circ \sigma\circ \tau_{\mbf x}$:
\[
(\tau_\mbf{u} \circ \sigma \circ \tau_{\mbf x}) (\mbf x) = \tau_{\mbf u} (\sigma(\mbf 0)) =  \tau_{\mbf u} (\mbf 0) = \mbf u ;
\]
\[
(\tau_\mbf{u} \circ \sigma\circ \tau_{\mbf x})(\mbf y) = \tau_{\mbf u} (\sigma(\mbf x + \mbf y)) =   \tau_{\mbf u} (\mbf u + \mbf w) = \mbf w.
\]

Thus the automorphism $\tau_\mbf{u} \circ \sigma\circ \tau_{\mbf x}$ maps the arc $(\mbf x, \mbf y)$ to $(\mbf u, \mbf w)$ proving that $\Omega_{2k}$ is arc-transitive.  The edge- and vertex-transitivity of $\Omega_{2k}$ follow from its arc-transitivity. 

\section{$\Det(\Omega_{2k})$ and a Lower Bound on $\Dist(\Omega_{2k})$}\label{sec:lowerbd}

To approach the determining number and distinguishing number of $\Omega_{2k}$ we will first want to study the twin vertices in $\Omega_{2k}$.

\begin{lemma} \label{lem:twins}\rm The vertices of $\Omega_{2k}$ can be partitioned into twin pairs of the form $\{\mbf u, \mbf u + \one\}$ for $\mbf u \in \Omega_{2k}$. In particular, $\mbf u$ and $\mbf w$ are twins if and only if $\mbf w = \mbf u +\mbf 1$.  Further,  there is no set of vertices of size three or more which are pairwise twins. \end{lemma}

\begin{proof} Note that the Hamming distance between $\bu$ and $\bu+\one$ is $2k$, so $\bu$ and $\bu + \one$ are nonadjacent.  Suppose $\bv \in N(\bu)$.  Then ${\rm wt}(\bv + \bu) = |{\rm supp}(\bv + \bu)| = k$, and  ${\rm wt}((\bv+\bu)+\one)=|{\rm supp}((\bv+\bu)+\one)|=2k-k=k$.  Thus $\bv \in N(\bu+\one)$.  Similarly, $\bv \in N(\bu + \one) $ means that $\bv \in N(\bu)$. Thus $N(\bu)=N(\bu+\one)$ so $\bu$ and $\bu+\one$ are twins.  
 
Suppose that $\mbf w \ne \bu+\one$ and that $\mbf w$ is not adjacent to $\mbf u$. We will show that there is some $\mbf y \in N(\mbf w)$ so that $\mbf y \not \in N(\bu)$.  
 
Let ${\rm wt}(\mbf w+\bu) = \ell$. Since $\mbf w$ is not adjacent to $\mbf u$, $\ell \ne k$. Let $r$ be the smaller of $\ell$ and $k$.  Choose $\mbf x \in V(\Omega_{2k})$ of weight $k$ so that its support overlaps with $r$ positions in the support of  $\mbf w + \mbf u$.  Let $\mbf y= \mbf w + \mbf x$. Since ${\rm wt}(\mbf x)=k$, $\mbf y \in N(\mbf w)$.  By our choice of support for $\mbf x$, ${\rm wt}(\mbf y + \mbf u)  = {\rm wt}((\mbf w + \mbf x) + \bu)=  {\rm wt}((\mbf w + \mbf u) + \mbf x) = |k-\ell|$.  

Further,  since $\ell \not\in \{0,2k\}$, we get that ${\rm wt}((\mbf w + \mbf u)+\mbf x) \ne k$.  Thus $\mbf y\not\in N(\mbf u)$.  Thus $\mbf w$ and $\mbf u$ are not twins.  In particular, each vertex $\mbf u$ in $\Omega_{2k}$ has a unique twin, $\mbf u + \one$.  	

Thus the vertices of $\Omega_{2k}$ can be partitioned into  twin pairs, and these pairs have the form $\{\mbf u, \mbf u + \one\}$.\end{proof}

Together Lemma \ref{lem:DisDetTwins} and Lemma \ref{lem:twins} prove the following.

\begin{theorem}\label{cor:DetOmega} \rm A subset of $V(\Omega_{2k})$ is a  determining set for $\Omega_{2k}$ if and only if it contains at least one vertex from each twin pair.  Thus $\Det(\Omega_{2k}) = 2^{2k-1}$.
\end{theorem}

The following lemma helps us understand how automorphisms of $\Omega_{2k}$ interact with twin pairs.

\begin{lemma}\label{lem:AutoPolar} \rm 
Any automorphism $\alpha \in {\rm Aut(}\Omega_{2k})$ preserves twin pairs. That is, for all $\mbf u \in V(\Omega_{2k})$,
$
\alpha(\mbf u + \mbf 1) = \alpha(\mbf u) + \mbf 1.
$
\end{lemma}

\begin{proof}
Since automorphisms preserve adjacency and nonadjacency, $\alpha(\mbf u)$ and $\alpha(\mbf u + \mbf 1)$ must be nonadjacent vertices with exactly the same set of neighbors. By Lemma~\ref{lem:twins}, the only other vertex with exactly the same neighbors as $\alpha(\mbf u)$ is  its twin $\alpha(\mbf u) + \mbf 1$.\end{proof}

Our knowledge of twin pairs helps us prove in the following that $\Omega_{2k}$ is not 2-distinguishable.

\begin{theorem}\label{thm:Dist2} \rm 
 {\rm Dist(}$\Omega_{2k}) > 2$.  \end{theorem} 

\begin{proof} 
Suppose there exists a distinguishing 2-labeling $f$  of $\Omega_{2k}$.  We will call the labels red and green.  Since twin vertices must get different labels, exactly half the vertices are red and exactly half the vertices are green.  For any vertex $\mbf u$, let $\pi_{\mbf u}$ be the automorphism that interchanges $\mbf u$ and its twin $\bu + \one$, and leaves all other vertices fixed.  By Lemma 2, $f \circ \pi_{\mbf u}$ is also a distinguishing 2-labeling.  For each $\mbf u$ in our red label class that does not have a 1 in its first bit, apply $\tau_{\mbf u}$.  This process leads us to a distinguishing 2-labeling, $f'$ in which the red label class is precisely the set of vertices with a 1 in their first bit.  
  
Let $\sigma$ be the cyclic permutation $\big (2\  3  \, \cdots  \, (2k) \big) \in S_{2k}$. The corresponding permutation automorphism  is nontrivial and fixes the first bit of each vertex.  Thus $\sigma$ preserves the label classes. Hence the labeling $f'$ is not distinguishing, and thus by Lemma \ref{lem:EquivDist}, $f$ is also not a distinguishing labeling.
\end{proof}

\section{Structure and $\wtO$} \label{sec:structure}

To find an upper bound on the distinguishing number, it will be useful to understand the structure of $\Omega_{2k}$. We start with work by Godsil and Newman in \cite{GoNe2008} which describes $\Omega_{2k}$ in terms of the parity of the weights of its vertices, concluding the following.

\bigskip\noindent{\bf \boldmath Structure of $\Omega_{2k}$:} \rm
Each vertex of $\Omega_{2k}$  has either even or odd Hamming weight. Thus the vertices of $\Omega_{2k}$ can be partitioned into the set of even vertices and the set of odd vertices. 

\begin{enumerate}
\item If $k$ is even, then $\Omega_{2k}$ consists of two isomorphic connected components, namely the subgraph induced by the even vertices and  the subgraph induced by the odd vertices.  We will refer to these as the even component and the odd component respectively.

\item If $k$ is odd, then $\Omega_{2k}$ is connected and bipartite, with parts the set of even vertices and the set of odd vertices.
\end{enumerate}

To more fully understand the structure of a graph, it can be useful to study a quotient graph.  Given any equivalence relation $\sim$ on a vertex set, we define a corresponding quotient graph $G  \slash \hspace{-.05in}\sim$ whose vertices are the equivalence classes of vertices under $\sim$, with classes $[u]$ and $[w]$ being adjacent if there exist $u', w' \in V(G)$ with $u \sim u'$, $w \sim w'$ and $u'w' \in E(G)$. The quotient graph is smaller, possibly simpler, and yet preserves some structure of the original graph.

In $V(\Omega_{2k})$ we identify each vertex with its twin.  That is, we define $\mbf u \sim \mbf u + \one$.  It is easy to verify that this is an equivalence relation. We denote the resulting quotient graph by $\wtO$ and its vertices by $[\mbf u] = \{\mbf u , \mbf u+\mbf 1\}=[\mbf u + \mbf 1]$. Note that $\wtO$ has order $2^{2k-1}$ and is $\frac{1}{2}\binom{2k}{k}$-regular. 

\begin{example}\rm The quotient graph $\widetilde \Omega_2$ is  $K_2$.
\end{example}

\begin{example} \rm The quotient graph $\widetilde \Omega_4$ is a 3-regular graph of order 8, and consists of two isomorphic components. By degree considerations alone, $\widetilde \Omega_4$ is the disjoint union of two copies of $K_4$. \end{example}

Since for any $\mbf u \in V(\Omega_{2k})$, $\wt(\mbf u + \one) = 2k - \wt(\mbf u)$, we see that $\mbf u$ is an even vertex if and only if $\mbf u + \one$ is even.  Hence the vertices of the quotient graph can also be partitioned into even and odd vertices. Moreover, if $k$ is even, $\wtO$ also consists of an even and an odd component, and if $k$ is odd, $\wtO$ is also bipartite with an even and an odd part.

In the proof of the proposition that follows, we preview methods that will be used in Theorem \ref{thm:DistOmega} to find an upper bound on ${\rm Dist}(\Omega_{2k})$ using ${\rm Det}(\widetilde \Omega_{2k})$. 

\begin{proposition}\rm ${\rm Dist}(\widetilde\Omega_4) = 5$ and ${\rm Dist}(\Omega_4) = 4$.\end{proposition} 

\begin{proof} Each $K_4$ component of $\widetilde \Omega_4$ has distinguishing number 4; to distinguish between the two isomorphic components, we need 5 labels in total. Note that using 4 labels we can create $\binom{4}{2}=6$ distinct label-pairs and then use 5 of these pairs to label the vertices of $\widetilde{\Omega}_4$.  This provides a $5$-distinguishing labeling of $\widetilde{\Omega}_4$ with label-pairs, and extends naturally to a 4-distinguishing labeling of $\Omega_{4}$ with twin pairs in $\Omega_4$ assigned the labels from the pairs of assigned to vertices of $\widetilde\Omega_4$. See Figure~\ref{fig:Ortho4}.

Note that if its components cannot be 3-distinguished, then neither can $\Omega_4$.  Let $C$ be the component of even vertices of $\Omega_4$.  Recall from Example \ref{ex:O4}, that $C = C_8(1,2,3)$, so $V(C)$ consists of 4 fs with an edge between every pair of vertices that are not twins.  Suppose we label $C$ with 3 labels. Since there are precisely $\binom{3}{2}=3$ distinct label-pairs for the 4 distinct twin pairs, two twin pairs, say $\{\mbf u, \mbf u + \one\}$ and $\{\mbf w, \mbf w+\one\}$, are assigned the same pair of labels.  Without loss of generality we can assume that the labels on $\mbf u$ and $\mbf w$ are red and the labels on $\mbf u + \one$ and $\mbf w + \one$ are green (or replace $\mbf w$ with $\mbf w + \one$).  Let $\alpha$ be the vertex map of $C$ that transposes $\mbf u$ and $\mbf w$, transposes $\mbf u + \one$ and $\mbf w + \one$, and fixes all other vertices.  Since the complement of $C$ is a set of 4 disjoint edges between twin pairs, and since $\alpha$ transposes two of these edges, $\alpha$ is an automorphism of $\overline C$ and thus of $C$ itself. Further $\alpha$  preserves label classes. Thus this is not a 3-distinguishing labeling of $C$.  We conclude that there is no distinguishing 3-labeling for $C$ and therefore none for $\Omega_4$. Thus we have proved that $\Dist(\Omega_4) = 4$.
\end{proof}

\begin{figure}[htbp] 
   \centering
   \includegraphics[width=4in]{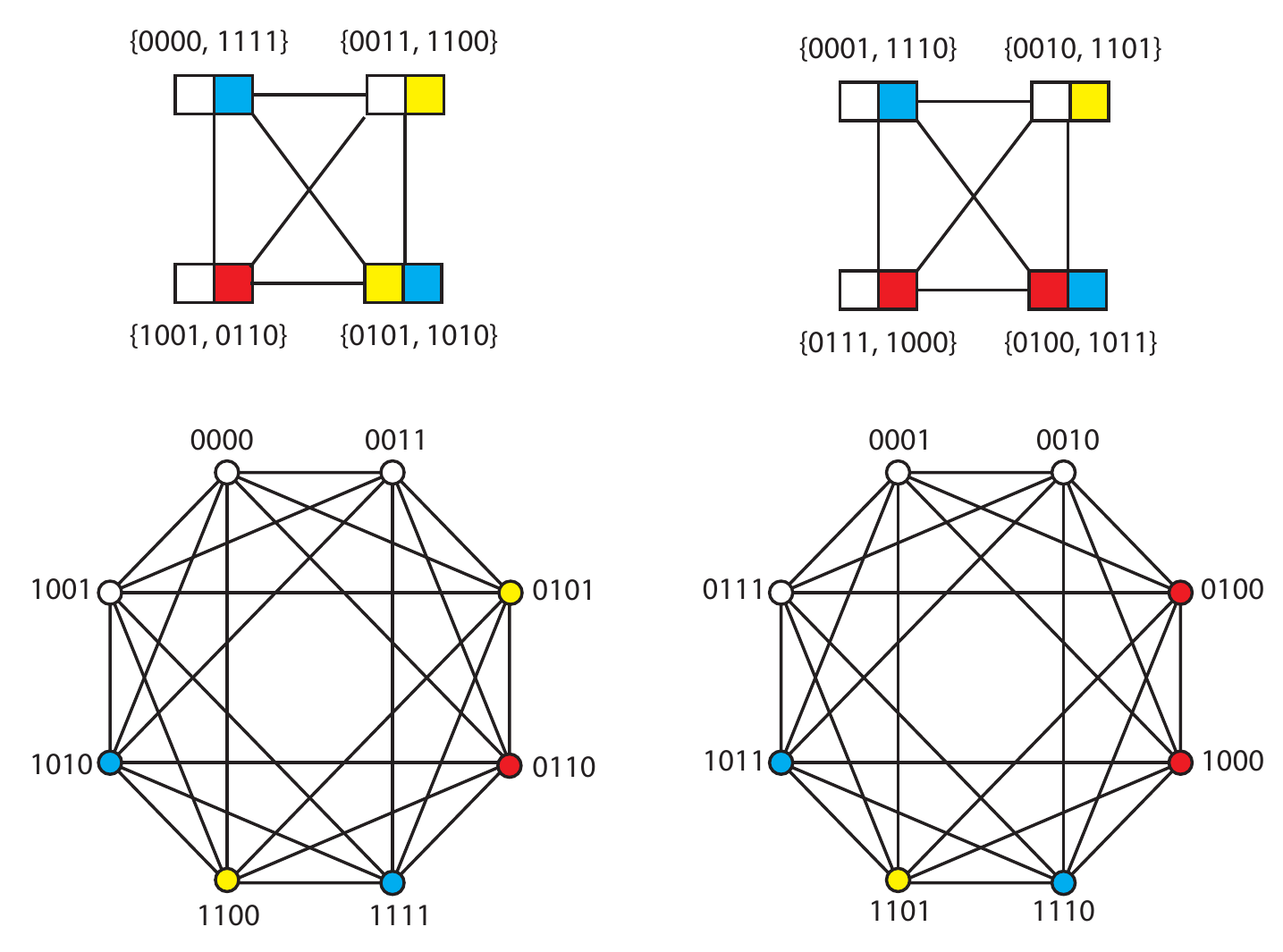} 
   \caption{$\widetilde{\Omega}_4$ with a 5-distinguishing labeling and $\Omega_4$ with a 4-distinguishing labeling }
   \label{fig:Ortho4}
\end{figure}

Next we will look more carefully at adjacencies within, and automorphisms of, $\wtO$. By Lemma~\ref{lem:twins},
\begin{align}\label{equiv:twins}
 \mbf u \text{ and } \mbf x \text{ are adjacent} \notag
& \iff \mbf u \text{ and } \mbf x + \mbf 1 \text{ are adjacent}\\ \notag
& \iff \mbf u + \mbf 1 \text{ and } \mbf x \text{ are adjacent}\\
& \iff \mbf u + \mbf 1 \text{ and } \mbf x +\mbf 1 \text{ are adjacent}.
\end{align}
This gives a stronger interpretation of the adjacency of $[\mbf u]$ and $[\mbf x]$ than is prescribed in the definition of a quotient graph. One implication of this is given below.

\begin{lemma}\label{lem:TildeTwinFree}\rm $\wtO$  is twin-free.
\end{lemma}

\begin{proof}
Suppose $N([\mbf v]) = N([\mbf u])$ in $\wtO$. Then for all $\mbf w \in V(\Omega_{2k})$,
\[
\mbf w \in N(\mbf v) \iff [\mbf w] \in N([\mbf v]) \iff [\mbf w] \in N([\mbf u]) \iff \mbf w \in N(\mbf u).
\]
Hence $N(\mbf v) = N(\mbf u)$.  By Lemma~\ref{lem:twins}, $\mbf v = \mbf u + \mbf 1$ in $\Omega_{2k}$, which means that $[\mbf v] = [\mbf u]$ in $\wtO$.
\end{proof}

Now let us look at the automorphism group.  By Lemma~\ref{lem:AutoPolar}, we can define a homomorphism  $\phi:\Aut(\Omega_{2k}) \to \Aut(\wtO)$ by $\phi(\alpha)([\mbf u]) = [\alpha(\mbf u)]$ for all $[\mbf u] \in V(\wtO)$.
For any $\widetilde \beta\in \Aut(\wtO)$, define $\beta \in \Aut(\Omega_{2k})$ by arbitrarily designating one vertex in each twin pair with a subscript of $0$, and defining $\beta(\mbf w_0) = (\widetilde \beta([\mbf w_0])_0$ and $\beta(\mbf w_0 + \mbf 1) = \beta(\mbf w_0) + \mbf 1.$
This shows that $\phi$ is surjective.
 The kernel of $\phi$ consists of all automorphisms of $\Omega_{2k}$ that simply interchange the vertices in some subset of the twin pairs. 
 More precisely, recall that $\pi_{\mbf u}$ is the automorphism of $\Omega_{2k}$ that interchanges $\mbf u$ and  $\mbf u + \mbf 1$ while fixing all other vertices. 
 There are $|V(\wtO)|= 2^{2k-1}$ such automorphisms  
 and they commute pairwise.  
 Let $U=\{[\mbf u_1], \ldots, [\mbf u_n]\} \subseteq V(\wtO)$, and let $\pi_U$ denote the composition $\pi_{\mbf u_1}\circ \cdots \circ \pi_{\mbf u_n}$. Then 
 $ \ker(\phi) = \{ \pi_U \mid U \subseteq V(\wtO)\} \cong (\mathbb Z_2)^{2^{2k-1}}$ and 
$
 \Aut(\wtO)  \cong \Aut(\Omega_{2k}) / (\mathbb Z_2)^{2^{2k-1}}.
$

\section{Determining $\wtO$} \label{sec:DetQuotient}

\begin{definition} \rm For any $[\mbf x] \in V(\wtO)$, we define 
\[
\wt([\mbf x]) = (\wt( \mbf x) , \wt(\mbf x + \mbf 1)) = (\wt(\mbf x), 2k-\wt(\mbf x)).
\]\end{definition}  

To eliminate ambiguity,  we assume $\wt(\mbf x) \leq k$ and thus that $\wt(\mbf x)\leq 2k-\wt(\mbf x)$. For example, $\{0010 1100,  1101 0011\} \in V(\widetilde\Omega_8)$ has weight $(3, 5)$.

In what follows, we will be concentrating on the odd vertices in $\wtO$.  Our goal is to show that every odd vertex in $\wtO$ has a unique set of neighbors among the set of vertices of weight $(k-1, k+1)$.

In the original orthogonality graph, let $\mbf u$ be a vertex of odd weight $m$, with $1 < m \leq k$. Let $\mbf v$ be a vertex of weight $k$.  Then $\mbf u + \mbf v$ is a neighbor of $\mbf u$. 
If  $|\supp(\mbf u) \cap \supp(\mbf v)| =t$,  then the  neighbor $\mbf u + \mbf v$ of $\mbf u$ has weight exactly $m+k-2t$. See Figure \ref{fig:WeightNeighbor}.

\begin{figure}[htbp] 
   \centering
   \includegraphics[width=2.5in]{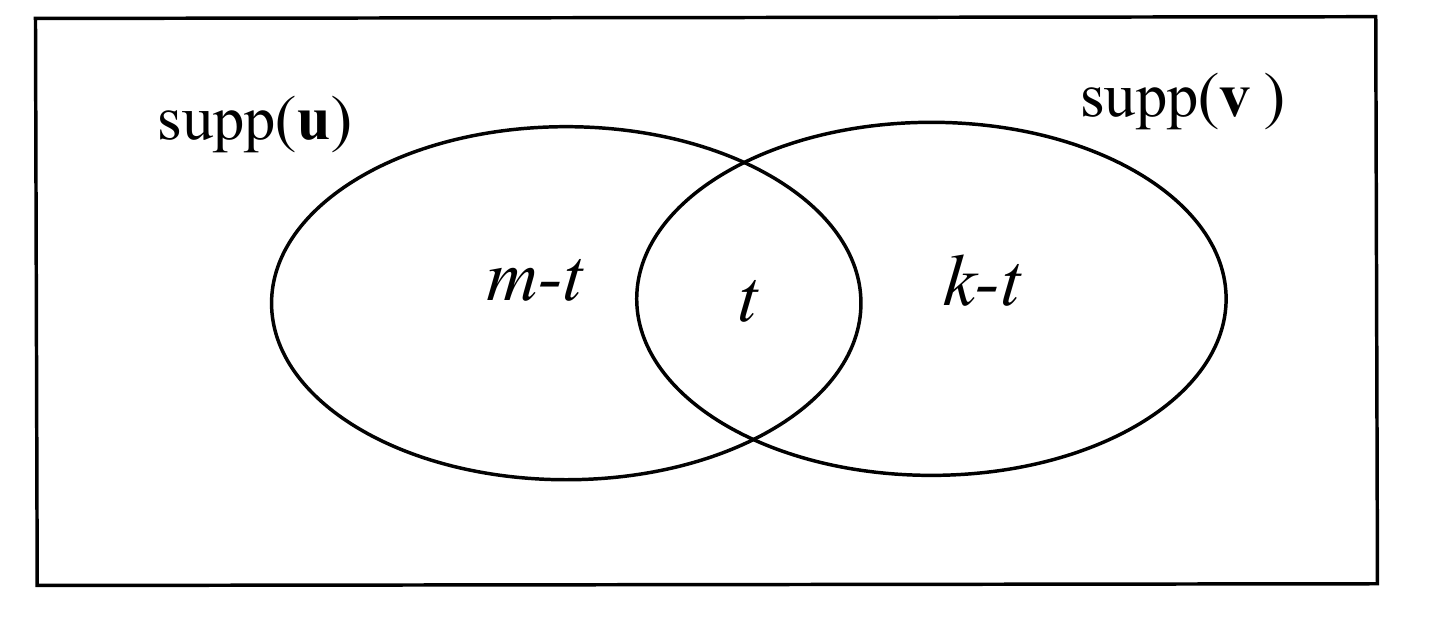} 
   \caption{The weight of a neighbor of $\mbf u$ is $m + k - 2t$.}
   \label{fig:WeightNeighbor}
\end{figure}

The number of neighbors of $\mbf u$ of weight $m+k-2t$ is the number of $\mbf v$ such that $|\supp(\mbf u) \cap \supp(\mbf v)| =t$, which is 
\[
\binom{m}{t} \binom{2k-m}{k-t}.
\]
Now, $m+k-2t = k-1 \iff t = \frac{m+1}{2}$ and $m+k-2t = k+1 \iff t= \frac{m-1}{2}$.
Using binomial identities,
\[
\binom{m}{\frac{m+1}{2}} \binom{2k-m}{k-\frac{m+1}{2}}= \binom{m}{\frac{m-1}{2}} \binom{2k-m}{k-\frac{m-1}{2}}.
\]
This makes sense because, by Lemma~\ref{lem:twins}, the neighbors of $\mbf u$ of weight $k+1$  and of weight $k-1$ can be matched up into twin pairs. Thus, in $\wtO$, the number of neighbors of $[\mbf u]$ of weight $(k-1, k+1)$  is the common value of the expression in the equation above.

\medskip

Note that the theorems and propositions that follow are written for $k\geq 1$.  However, because their statements involve an odd integer  $m$ with $1<m\leq k$,  for Lemma \ref{lem:Binom} and Corollary \ref{cor:DiffWeights} technically $k\geq 3$, while for Lemma \ref{lem:mLessk} and Corollary \ref{cor:SameWeight} technically $k\geq 4$.   This does not change the fact that all theorems are true for all $k\geq 1$. 

\begin{lemma} \label{lem:Binom} \rm 
For distinct odd $m, n$, both less than or equal to $k$,  
\[
\binom{m}{\frac{m+1}{2}} \binom{2k-m}{k-\frac{m+1}{2}} \neq  \binom{n}{\frac{n+1}{2}} \binom{2k-n}{k-\frac{n+1}{2}}.
\]
\end{lemma}

The proof is by binomial computation and is contained in Appendix A. 

\begin{corollary}\label{cor:DiffWeights} \rm 
For distinct odd $m, n$, both less than or equal to $k$,  vertices in $\wtO$ of weight $(m, 2k-m)$ have a different number of neighbors of weight $(k-1, k+1)$ than vertices in $\wtO$ of weight $(n, 2k-n)$.
\end{corollary}

We next consider distinct odd vertices in $\wtO$ of the same weight.  We start with a technical lemma about vertices in the original orthogonality graph $\Omega_{2k}$.

\begin{lemma} \label{lem:mLessk}\rm 
Let  $1< m < k$, with $m$  odd.  Let $\mbf u$ and $\mbf w$ be distinct  vertices in $\Omega_{2k}$ with $\wt(\mbf u) = \wt(\mbf w)= m$. Then there exists  $\mbf y \in V(\Omega_{2k})$ with $\wt(\mbf y) = k-1$ that is adjacent to $\mbf u$ but not to $\mbf w$.
 \end{lemma}
 
\begin{proof}We divide into two cases.

{\bf Case 1.} Assume $|\supp(\mbf u) \cap \supp(\mbf w)| = 0$.  To find a neighbor $\mbf y$  of $\mbf u$ with $\wt(\mbf y) = k-1$, we must find $\mbf v \in V(\Omega_{2k})$ of weight $ k$ such that $\mbf y = \mbf u + \mbf v$ satisfies
\[ 
\wt(\mbf y) = \wt(\mbf u + \mbf v) = |\supp(\mbf u + \mbf v)| = |\supp(\mbf u)\,  \triangle \, \supp(\mbf v) | = k-1.
\]

We can construct such a $\mbf v$ by selecting  $\frac{m+1}{2}$ positions from each of the disjoint sets $\supp(\mbf u)$ and $\supp(\mbf w)$.  This accounts for $m+1$ of the necessary $k$ positions in $\supp(\mbf v)$.  Note that $m<k \Longrightarrow m+1 \leq k$.  Thus we need to add (the nonnegative number) $k-(m+1)$ positions from the $2k-2m$ positions outside both $\supp(\mbf u) $ and $ \supp(\mbf w)$ to $\supp(\mbf v)$ to achieve $\wt(\mbf v)=k$. See Figure~\ref{fig:SuppDisjointv2} for a Venn diagram showing that this is achievable. Note that $|\supp(\mbf y)| = |\supp(\mbf u)\,  \triangle \, \supp(\mbf v)| = k-1,$ as desired. Moreover, since $
\wt(\mbf y + \mbf w) = \wt\big ( (\mbf u + \mbf v) + \mbf w \big )
$ and $|\supp(\mbf u)\, \triangle\, \supp(\mbf v)\, \triangle \,\supp(\mbf w)| = k-2 < k$, $\mbf y$ is not adjacent to $\mbf w$.

\begin{figure}[htbp] 
   \centering
   \includegraphics[width=2.7in]{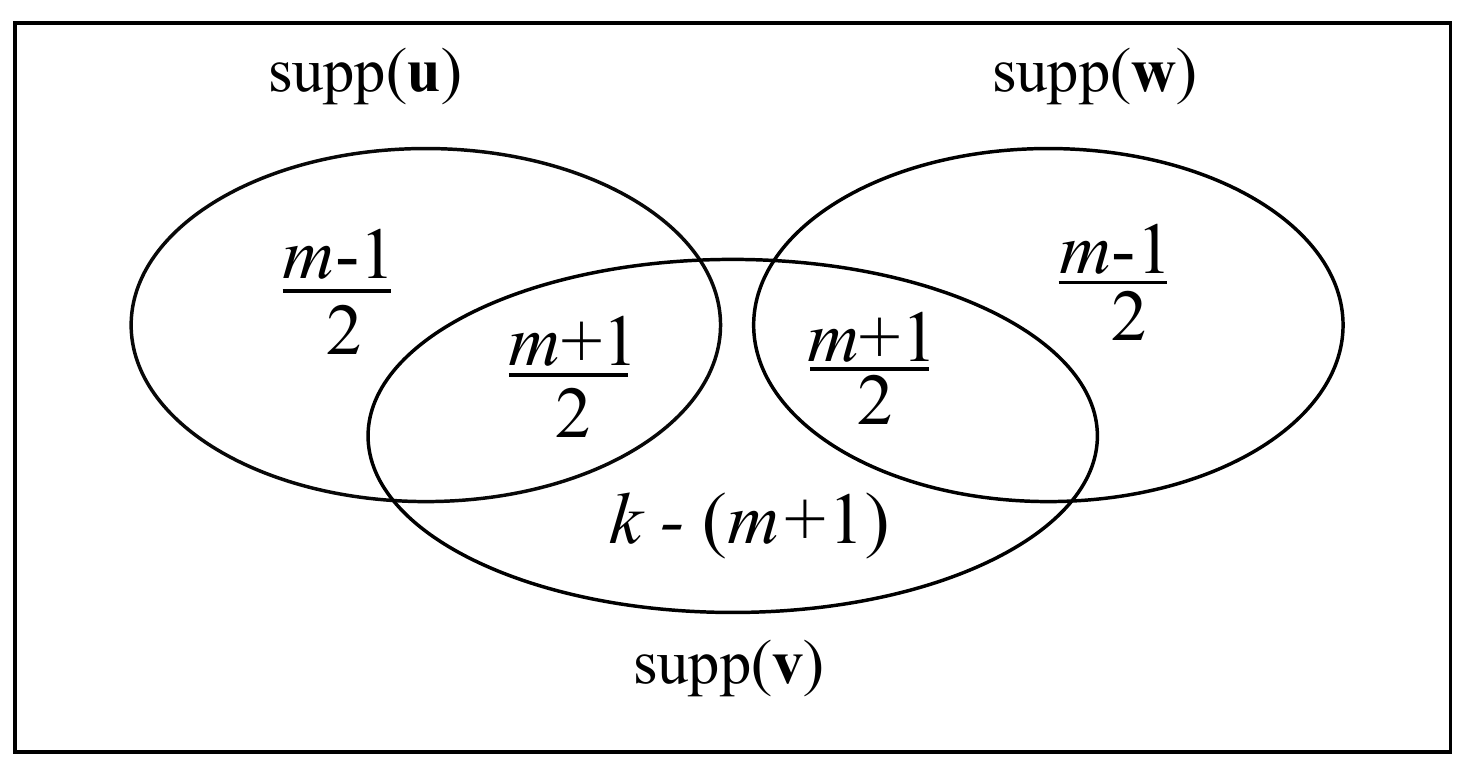} 
   \caption{Case 1: $|\supp(\mbf u) \cap \supp(\mbf w)| = 0$}
   \label{fig:SuppDisjointv2}
\end{figure}
 
{\bf Case 2.} Assume $ |\supp(\mbf u) \cap \supp(\mbf w)|= r \ge 1$. 
First suppose $r = 2b+1$ for some $b\ge 0$. Since $m$ is also odd,  $m = 2a+r = 2a + 2b+1$ for some $a \ge 1$. We can construct an appropriate $\mbf v$ by choosing positions for its support in the following way.  Choose $a$ positions from $\supp(\mbf u) \setminus \supp(\mbf w)$, $a$ positions from $\supp(\mbf w) \setminus \supp(\mbf u)$, $b$ positions from $\supp(\mbf w)\cap \supp(\mbf u)$, and $k-(2a+b+1)$ positions from the complement of $\supp(\mbf w) \cup \supp(\mbf u)$. Note that the number of positions outside $\supp(\mbf w) \cup \supp(\mbf u)$ is $2k - (4a+2b+1)> 2[k -(2a + b +1)]$, so we have plenty of positions from which to choose. See Figure~\ref{fig:SuppNonDisjointOdd}. It is easy to check that  $\mbf y = \mbf u + \mbf v$ has weight $k-1$
and that $\mbf y + \mbf w$ has weight $k-(b+1) \leq k-1 < k$, so $\mbf y$ is not adjacent to $\mbf w$. 

\begin{figure}[htbp] 
   \centering
   \includegraphics[width=2.7in]{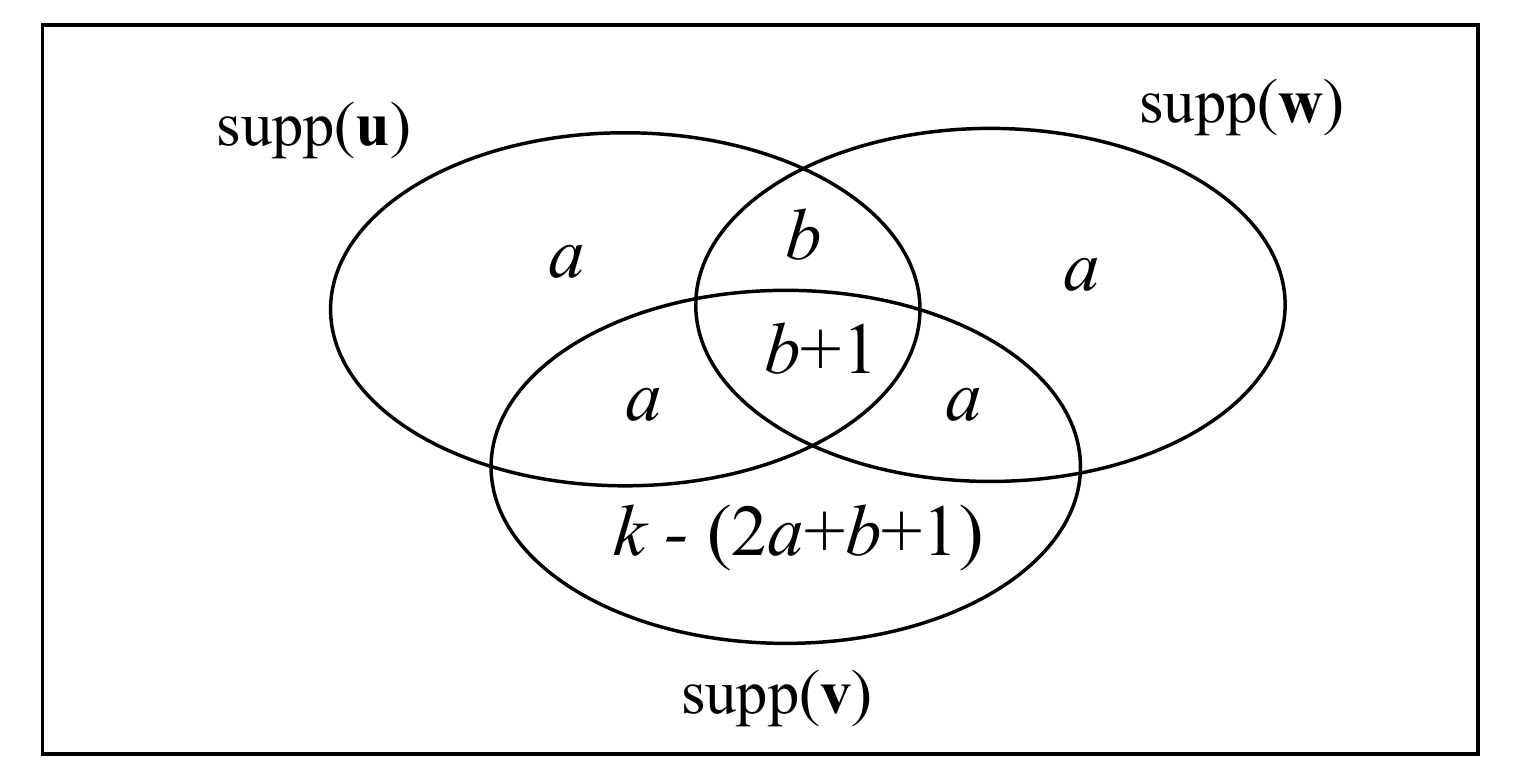} 
   \caption{Case 2: $r = |\supp(\mbf u) \cap \supp(\mbf w)|$ is odd.} 
   \label{fig:SuppNonDisjointOdd}
\end{figure}

Now suppose that $r = 2b$ for some $b \ge 1$. Since $m$ is odd, $m = (2a+1) +r = 2a+2b+1$ for some $a \ge 0$. We can find $\mbf v$, and therefore $\mbf y$, using Figure~\ref{fig:SuppNonDisjointEven}.  That is, we choose $a+1$ positions from $\supp(\mbf u) \setminus \supp(\mbf w)$, $a+1$ positions from $\supp(\mbf w) \setminus \supp(\mbf u)$, $b$ positions from $\supp(\mbf w)\cap \supp(\mbf u)$, and $k-(2a+b+2)$ from the complement of $\supp(\mbf w) \cup \supp(\mbf u)$. We can again easily verify that  $\mbf y = \mbf u + \mbf v$ has weight $k-1$ and that $\mbf y + \mbf w$ has weight $k-(b+2) \leq k-3 < k$, so $\mbf y$ is not a neighbor of $\mbf w$.\end{proof}

\begin{figure}[htbp] 
   \centering
   \includegraphics[width=2.7in]{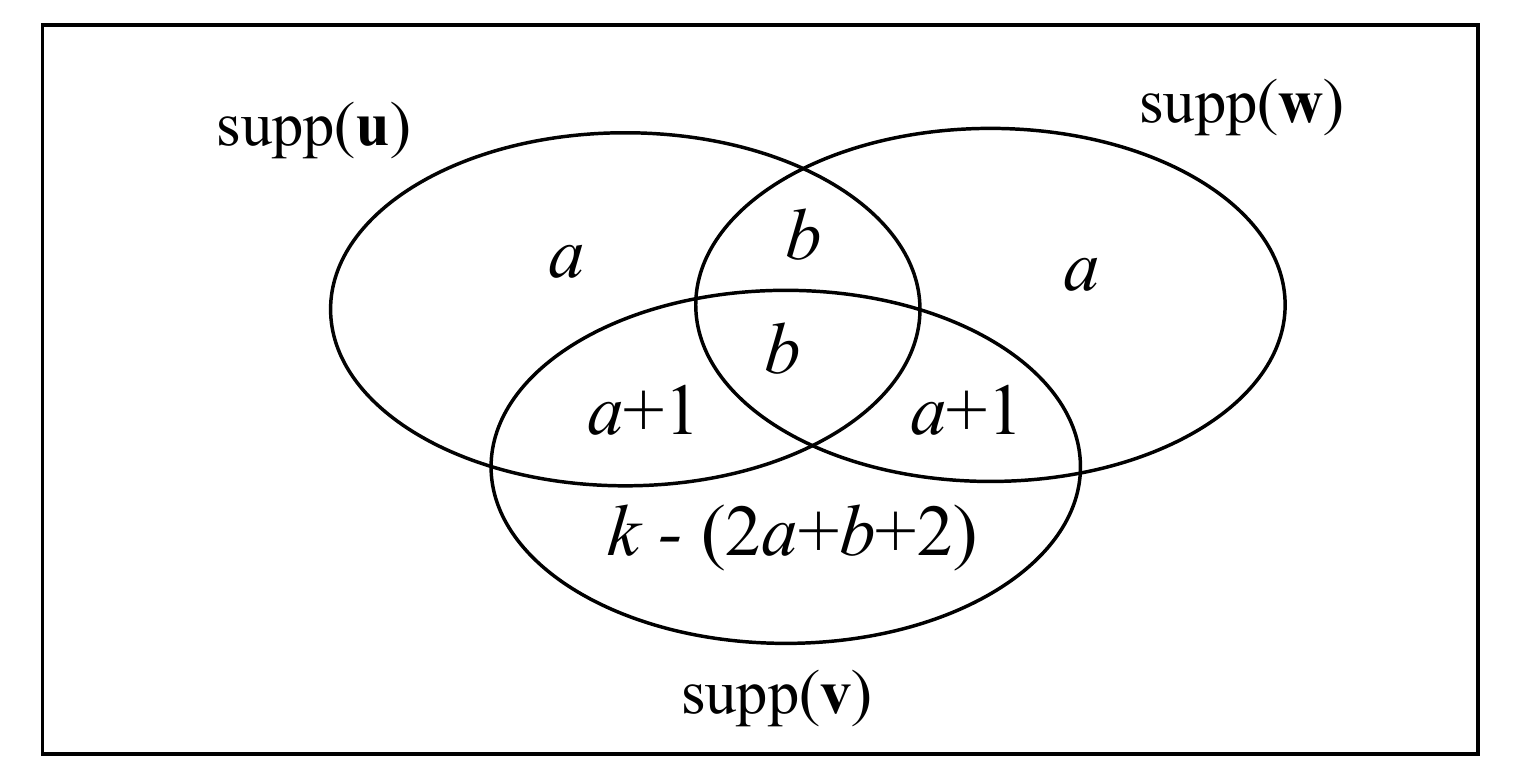} 
   \caption{Case 2: $r = |\supp(\mbf u) \cap \supp(\mbf w)|$ is even.} 
   \label{fig:SuppNonDisjointEven}
\end{figure}

By passing to the quotient graph, we have the following.
\newpage

\begin{corollary}\label{cor:SameWeight} \rm 
Let $1< m < k$, with $m$  odd.   Let $[\mbf u]$ and $[\mbf w]$ be distinct vertices in $\wtO$ of the same weight   $(m, 2k-m)$. Then there exists a vertex $[\mbf y]$ of weight $(k-1, k+1)$ that is adjacent to $[\mbf u]$ but not to $[\mbf w]$.   \end{corollary}

Combining Corollaries~\ref{cor:DiffWeights} and \ref{cor:SameWeight} achieves our goal.

\begin{proposition}\label{prop:UniqueNeighbors} \rm 
Each odd vertex in $\wtO$ has a unique set of neighbors among the set of vertices of weight $(k-1, k+1)$. 
\end{proposition} 

Now we can explicitly build a determining set for $\Omega_{2k}$ and $\wtO$.  For $i \in \{1,  \dots , 2k\}$, let $\mbf x_i$ denote the vertex of $\Omega_{2k}$ represented as a bitstring with a 1 in position $i$ and 0's elsewhere, with $[ \mbf x_i]$ being the corresponding vertex of $\widetilde\Omega_{2k}$. 

\begin{proposition}\label{prop:DetD} \rm 
Let  $D =  \{[\mbf x_1], [\mbf x_2], \dots, [\mbf x_{2k-1}]\} $, which is a subset of the odd vertices of $\wtO$.
If $k$ is even, then $D$ is a determining set for the odd component of $\wtO$.
If $k$ is odd,  then $D$ is a determining set for $\wtO$.
\end{proposition} 

\begin{proof}
Assume $\tilde \alpha \in \Aut(\wtO)$ fixes pointwise the vertices in $D$.  Any graph automorphism of $\wtO$ must respect  its separation into two components if $k$ is even, or its bipartition if $k$ is odd. Thus, since $\tilde \alpha$ fixes $D$, $\tilde \alpha$ must map odd vertices to odd vertices and even vertices to even vertices.

One can easily verify that every neighbor of a vertex in $D$ has weight $(k-1, k+1)$. Conversely, let $\mbf y \in V(\Omega_{2k})$ be a vertex of weight $k+1$. Then $\mbf y$ is adjacent to $\mbf x_i$ if and only if $i \in \supp(\mbf y)$; equivalently $\mbf y$ can be uniquely identified either by which $k+1$ of the $\mbf x_i$ it is adjacent to, or by which $k-1$ of the $\mbf x_i$ it is not adjacent to. In the quotient graph,
\[
\{ [\mbf y]\} = \bigcap \{N([\mbf x_i]) \mid i \in \supp(\mbf y)\}.
\]
If $2k \notin \supp(\mbf y)$, then $[\mbf y]$ is the unique common neighbor of $k+1$ elements of $D$.
If $2k \in \supp(\mbf y)$, then $[\mbf y]$ can still be identified by which $k$ elements of $D$ it is adjacent to and which $k-1$ elements it is not adjacent to.

Thus fixing $D$  fixes all vertices in $\wtO$ of weight $(k-1, k+1)$. Then by Proposition~\ref{prop:UniqueNeighbors}, $\tilde \alpha$ must fix every odd vertex of $\wtO$. If $k$ is even, then we are done.

If $k$ is odd, then $\wtO$ is bipartite with each even vertex having only odd neighbors.
By Lemma~\ref{lem:TildeTwinFree}, since $\wtO$ is twin-free, no two nonadjacent (i.e., even) vertices of $\wtO$ have the same neighborhood. Hence $\tilde \alpha$ also fixes all even vertices and we are done.\end{proof}

Although the preceding proposition does not assert that $D$ is a minimum size determining set,  it is a minimal determining set. Without loss of generality let  $D' =  \{[\mbf x_1], [\mbf x_2], \dots, [\mbf x_{2k-2}]\} $. 
Let $\sigma \in S_{2k} $ be the transposition permutation that interchanges $2k-1$ and $2k$. Then the corresponding nontrivial permutation automorphism on $\Omega_{2k}$ fixes $\mbf x_1, \dots \mbf x_{2k-2}$ and so the induced nontrivial automorphism  on $\wtO$ fixes the elements of $D'$. 

\begin{corollary}\label{cor:DetQuotient} ${\rm Det}(\widetilde{\Omega}_{2k}) \leq 2k-1$. \end{corollary}

\section{Distinguishing $\Omega_{2k}$}\label{sec:upperbd}

\begin{theorem}\label{thm:DistOmega} \rm 
 $2 < $ Dist$(\Omega_{2k}) \leq m$, where $m$ is the smallest integer that satisfies
\[
\binom{m}{2} \geq 
\begin{cases}
2k, \quad & k \text{ odd},\\
2k+1, & k \text{ even}.
\end{cases}
\]
\end{theorem}  

\begin{proof}
First assume $k$ is odd. By Proposition~\ref{prop:DetD}, $D$ is a determining set for $\wtO$. 
The subgraph of $\wtO$ induced by $D$ is a null graph and so has distinguishing number $|D| = 2k-1$.
Thus by Theorem~\ref{thm:distdet}, $\wtO$ can be $2k$-distinguished. 

Next assume $k$ is even. By Proposition~\ref{prop:DetD}, $D$ is a determining set for the odd component of $\wtO$. 
If $k=2$, then the subgraph of $\wtO$ induced by $D$ is a complete graph, and otherwise it is a null graph. In all cases, it has distinguishing number $2k-1$.
Thus by Theorem~\ref{thm:distdet}, the odd component of $\wtO$ can be $2k$-distinguished.  Since the even component is an isomorphic copy of the odd component,  we need only one more label to distinguish the even component and to distinguish it from the odd component.  Thus $\wtO$ can be $(2k+1)$-distinguished.

Suppose there exists an $\ell$-distinguishing labeling $\tilde f$  of $\wtO$. To extend it to a distinguishing labeling on $\Omega_{2k}$, recall that by Lemma~\ref{lem:DisDetTwins}, twin vertices in $\Omega_{2k}$ must be assigned different labels in any distinguishing labeling.  

If $m$ satisfies 
\[
\binom{m}{2} \geq \ell,
\]
then we can create $\ell$ different label-pairs from $m$ different labels. We assign these label-pairs to vertices in  $\wtO$ according to $\tilde f$, then randomly assign one label from each label-pair to the  members of the corresponding twin pair in $\Omega_{2k}$. 

The following argument shows that this creates an $m$-distinguishing labeling of $\Omega _{2k}$.  Suppose $\alpha \in \Aut(\Omega_{2k})$ satisfies 
$f(\mbf u) = f(\alpha (\mbf u))$ for all $\mbf u \in V(\Omega_{2k})$. Then by Lemma~\ref{lem:AutoPolar},
\[
f(\mbf u + \mbf 1) = f(\alpha(\mbf u + \mbf 1)) = f(\alpha(\mbf u) + \mbf 1),
\]
and so
\begin{align*}
\tilde f([\mbf u]) & =\{f(\mbf u), f(\mbf u + \mbf 1)\} \\ &= \{ f(\alpha (\mbf u)),  f(\alpha(\mbf u) + \mbf 1)\} \\& = \tilde f ([\alpha(\mbf u)]) = \tilde f(\tilde \alpha ([\mbf u]).
\end{align*}
By the assumption that $\tilde f$ is distinguishing, $\tilde \alpha$ is the identity on $\widetilde \Omega_{2k}$, which means that either $\alpha(\mbf u) = \mbf u$ or $\alpha(\mbf u) = \mbf u + \mbf 1$.  Since twin vertices have different labels under $f$ and $\alpha $ respects $f$, $\alpha$ must be the identity on $\Omega_{2k}$.
\end{proof}

The table below shows  minimum values of the upper bound $m$ for $2 \leq k \leq 18$.
\begin{table}[htp]
\begin{center}
\begin{tabular}{|c|c|c|c|c|c|c|c|c|c|c|c|c|c|c|c|c|c|} \hline
$k$ & 2 & 3& 4& 5& 6& 7 & 8 & 9 & 10 & 11 & 12 & 13 &14 & 15 & 16 & 17 & 18\\
\hline
$m$ & 4 & 4 & 5 & 5 & 6 & 6 & 7 & 7 & 7 & 8 & 8 & 8  &9 &9 &9 &9 &9\\
\hline
\end{tabular}
\end{center}
\label{default}
\end{table}
 
\section{Open Questions}\label{sect:questions}

\begin{question}\rm Is ${\rm Det}(\Omega_{2k})=2k-1$ or can it be smaller? \end{question}

\begin{question} \rm Let $k\geq 2$; let $m$ be the smallest integer so that 
$
\binom{m}{2}~\geq~
\begin{cases}
2k, \quad & k \text{ odd},\\
2k+1, & k \text{ even}.
\end{cases}
$.  For which  $k\geq 2$ does  ${\rm Dist}(\Omega_{2k}) = m$?\end{question}

\section{Acknowledgments}

The authors wish to thank two anonymous referees for their careful reading and thoughtful comments on the manuscript that led to this article.

\begin{appendix}

\section{Proof of Lemma \ref{lem:Binom}}

\1{\bf Lemma \ref{lem:Binom}} 
For distinct odd $m, n$, both less than or equal to $k$,  
\[
\binom{m}{\frac{m+1}{2}} \binom{2k-m}{k-\frac{m+1}{2}} \neq  \binom{n}{\frac{n+1}{2}} \binom{2k-n}{k-\frac{n+1}{2}}.
\]

\begin{proof}
It suffices to show that the sequence 
\[
\binom{3}{2} \binom{2k-3}{k-2} , \binom{5}{3} \binom{2k-5}{k-3} , \dots, \binom{m}{\frac{m+1}{2}} \binom{2k-m}{k-\frac{m+1}{2}} 
\]
is monotone decreasing (where $m$ is largest odd number  satisfying $m \leq k$), and for this it suffices to show that for $n$ odd, $1 < n \leq k$,
\[
\binom{n-2}{\frac{n-1}{2}} \binom{2k-n+2}{k-\frac{n-1}{2}} >  \binom{n}{\frac{n+1}{2}} \binom{2k-n}{k-\frac{n+1}{2}}.
\]
We use some combinatorial algebra to rewrite the binomial coefficients:
\begin{align*}
\binom{n}{\frac{n+1}{2}} & = \frac{n!}{(\frac{n+1}{2})!(\frac{n-1}{2})!} \\
 & = \frac{n (n-1) (n-2)!}{(\frac{n+1}{2})( \frac{n-1}{2})!\, (\frac{n-1}{2}) (\frac{n-3}{2})!} \\
& = \frac{n(n-1)}{(\frac{n+1}{2})(\frac{n-1}{2})} \binom{n-2}{\frac{n-1}{2}} = \frac{4n}{n+1} \binom{n-2}{\frac{n-1}{2}}.
\end{align*}
Similar algebraic manipulations yield
\[
\binom{2k-n+2}{k-\frac{n-1}{2}}  = \frac{4(2k-n+2)}{2k-n+3} \binom{2k-n}{k-\frac{n+1}{2}}.
\]
Substituting in, we are trying to show that
\[
\binom{n-2}{\frac{n-1}{2}} \left [ \frac{4(2k-n+2)}{2k-n+3} \binom{2k-n}{k-\frac{n+1}{2}} \right ]
> \left [\frac{4n}{n+1} \binom{n-2}{\frac{n-1}{2}}\right ] \binom{2k-n}{k-\frac{n+1}{2}}.
\]
Canceling equal terms and cross-multiplying, this  holds if and only if
\[
(2k-n+2)(n+1) > n(2k-n+3),
\]
which simplifies to $k+1 > n$. Since we assumed $n \leq k$, we are done.
\end{proof}

\end{appendix}

\bibliographystyle{plain}
\bibliography{OrthogonalityGraphsJGT2}

\end{document}